\documentclass{amsart}

\usepackage{amssymb, amsmath, amsthm, units}
\usepackage{verbatim}
\usepackage{comment}
\usepackage{graphicx}
\usepackage{enumerate}
\usepackage{color}
\usepackage{hyperref}
\usepackage{cleveref}
\usepackage{fullpage}

\DeclareGraphicsExtensions{.pdf,.jpg,.png}

\theoremstyle{plain}
\newtheorem{theorem}{Theorem}
\newtheorem{lemma}[theorem]{Lemma}

\newtheorem{conjecture}[theorem]{Conjecture}

\theoremstyle{definition}

\newcommand{\bZ}{\mathbb{Z}}

\makeatletter

\begin{document}

\title{On the existence of dense substructures in finite groups}

\author{Ching Wong}
\address{\noindent Department of Mathematics, University of British Columbia, Vancouver, BC, Canada V6T 1Z2}
\email{ching@math.ubc.ca}

\date{}

\begin{abstract}
	Fix $k \geq 6$. We prove that any large enough finite group $G$ contains $k$ elements which span quadratically many triples of the form $(a,b,ab) \in S \times G$, given any dense set $S \subseteq G \times G$. The quadratic bound is asymptotically optimal. In particular, this provides an elementary proof of a special case of a conjecture of Brown, Erd\H{o}s and S\'{o}s. We remark that the result was recently discovered independently by Nenadov, Sudakov and Tyomkyn.
\end{abstract}

\maketitle
\section{Introduction}
This note is devoted to the study of a special case of a long-standing conjecture of Brown, Erd\H{o}s and S\'{o}s \cite{BrownErdosSos73} in 1976, which was originally formulated as a hypergraph extremal problem, and was found equivalent to the following by Solymosi, see \cite{Solymosi15}. 

\begin{conjecture}[Brown-Erd\H{o}s-S\'{o}s \cite{BrownErdosSos73}]
	\label{conj:BES}
	Fix $k \geq 6$. For every $c>0$, there exists a threshold $N = N(c)$ such that if $A$ is a finite quasigroup of order larger than $N$, then for every $S \subseteq A \times A$ with $|S| \geq c|A|^2$ where $c>0$, there exists a set of $k$ elements of $A$ which spans at least $k-3$ triples of the form $(a,b,ab) \in S \times A$.
\end{conjecture}

This conjecture is proved true only when $k = 6$, by Ruzsa and Szemer\'{e}di \cite{RuzsaSzemeredi1978} in 1978. The problem turns out to be delicate and remains unapproachable over decades in its full generality, for all $k \geq 7$.

It is desirable to look for subfamilies of quasigroups for which the conjecture is valid and a natural candidate is groups, where associativity holds. By exploiting this additional structure, the groundbreaking result of Solymosi \cite{Solymosi15} shows the validity of the conjecture when $k=7$ for finite groups.

Using the regularity lemma, Solymosi and the author \cite{SW} extended the result of \cite{Solymosi15} to $k=11,12$ and $k \geq 15$. Surprisingly, instead of the conjectured lower bound $k-3$, we found asymptotically $4k/3$ triples spanned by a set of $k$ elements. One is then led to the question: For finite groups, what is the right magnitude of the maximum number of triples spanned by $k$ elements?

In this note we give an elementary proof that such lower bound is quadratic in $k$, matching the trivial upper bound $k^2$. An immediate consequence is an alternate proof of \cref{conj:BES} for finite groups when $k$ is large.

\begin{theorem}
	\label{thm:main}
	Let $k \geq 6$ be an integer, then there exists a threshold $N = N(k)$ such that if $G$ is a finite group of order larger than $N$, then for every $S \subseteq G \times G$ with $|S| \geq c|G|^2$ where $c > 0$, there exists a set of $k$ elements of $G$ which spans at least $\frac{c}{2^{10}}k^2$ triples $(a,b,ab)$ from $S$, i.e. $(a,b) \in S$.
\end{theorem}

A stronger result, where the coefficient of $k^2$ is independent of $c$, was recently discovered independently by Nenadov, Sudakov, and Tyomkyn.

The rest of this note is dedicated to the proof of \cref{thm:main}. The main ingredient is the construction of explicit subsets whose arbitrary two-fold products are highly structured. This is demonstrated in \cref{sec:idea} for the model groups $\bZ_n$ and $\bZ_m^n$, which are respectively large in the order and the exponent.

The case of general abelian groups can be readily reduced to these model groups, using the Fundamental Theorem of Finite Abelian Groups, which states that every finite abelian group is a direct product of cyclic groups. Finally, the argument can be carried over to arbitrary finite groups by considering cosets of the form $\ell H$, $H r$ and $\ell H r$, where $H$ is a large abelian subgroup whose existence is guaranteed by a theorem of Pyber's \cite{PY}. This is the content of \cref{sec:proof}.

\section{Idea of the proof}
\label{sec:idea}
Fix $k \geq 1$. Let $G$ be a finite (abelian) group. For every element $x \in G$, we will construct a subset $A_x$ of $G$ with the following properties:

\begin{enumerate}
	\item The set $A_x$ has size $k$.
	
	\item The set $A_xA_y := \{ab \in G : a \in A_x, b \in A_y\}$ has size at most $2k$.
	
	\item If $x \neq y$, then $A_x \neq A_y$.
	
	\item Any $(a,b,ab) \in G^3$ is contained in $k^2$ many $B_{x,y}$'s, where
	\[
	B_{x,y} := \{ (a,b,ab) \in G^3 : a \in A_x, b \in A_y \}.
	\]	
\end{enumerate}

Then, by pigeonhole principle, together with (3) and (4), there exists $B_{x,y}$ that contains at least $|S| k^2 / |G|^2$ triples from $S$, for any $S \subseteq G \times G$. Note that the triples in $B_{x,y}$ are spanned by the elements $A_x \cup A_y \cup A_xA_y$, which has size at most $4k$, by condition (1) and (2). The assumption that $|S| \geq c|G|^2$ implies that there is a set of $4k$ elements of $G$ which spans at least $ck^2$ triples from $S$.

To fix the idea, we consider the model cases $\bZ_n$ and $\bZ_m^n$.

\subsection{When $G \cong \bZ_n$ is a cyclic group}
\label{sec:cyclicGroups}
Fix $k \geq 1$ and let $n \geq 4k$. For $x \in \bZ_n$, let
\[
A_x = \{ x+i \in \bZ_n : 0 \leq i \leq k-1 \}
\]
be a set of $k$ elements of $\bZ_n$. We check the other 3 conditions one by one. The set
\[
A_x + A_y = \{ a + b : a \in A_x, b \in A_y \} = \{ x+y+j \in \bZ_n : 0 \leq j \leq 2k-2 \}
\]
has $2k-1$ elements. It is clear that different elements $x,y \in \bZ_n$ yield different sets $A_x$, $A_y$. To see that the last condition holds, one can choose $x$ from the set $\{a, a-1, \ldots, a-(k-1)\} \subseteq \bZ_n$, and choose $y$ from the set $\{b, b-1, \ldots, b-(k-1)\} \subseteq \bZ_n$, given any $(a,b,a+b) \in \bZ_n^3$.

\subsection{When $G \cong \bZ_m^n$, for some $m \geq 2$}
\label{sec:Zmn}
Fix integers $\rho \geq 1$, $1 \leq t < m$ and let $n$ be large such that $m^n \geq 4tm^\rho$. (We suppose for now that $k = tm^\rho - 1$.) For $x = (x_1,\ldots,x_n) \in \bZ_m^n$, let
	\[
	A_x = \{ (z_1,\ldots,z_{\rho},x_{\rho+1}+i,x_{\rho+2},\ldots,x_n) \in \bZ_m^n : z_i \in \bZ_m, 0 \leq i \leq t-1\} \backslash \{x\}
	\]
	be a set of $tm^\rho - 1$ elements of $\bZ_m^n$. The set
	\[
	\begin{split}
	& \quad \,\, A_x+A_y = \{ a + b : a \in A_x, b \in A_y \} \\
	& \subseteq \{ (z_1,\ldots,z_{\rho}, x_{\rho+1} + y_{\rho+1} + j, x_{\rho+2} + y_{\rho+2}, \ldots, x_n+y_n) \in \bZ_m^n : z_i \in \bZ_m, 0 \leq j \leq 2t-2 \}
	\end{split}
	\]
	has size at most $(2t-1)m^\rho \leq 2(tm^\rho-1)$. Since $x \not\in A_x$, it is easy to see that different elements $x,y \in \bZ_m^n$ yield different sets $A_x$, $A_y$. Finally, given $(a,b,a+b) \in (\bZ_m^n)^3$, the number of $B_{x,y}$ that contains $(a,b,a+b)$ is $(tm^\rho-1)^2$. Indeed, one can choose $x$ from the set
	\[
	\{ (z_1,\ldots,z_{\rho}, a_{\rho+1} - i, a_{\rho+2}, \ldots, a_n) \in \bZ_m^n : z_i \in \bZ_m, 0 \leq i \leq t-1 \} \backslash \{a\}.
	\]
	and choose $y$ from the set
	\[
	\{ (z_1,\ldots,z_{\rho},b_{\rho+1} - i, b_{\rho+2}, \ldots, b_n) \in \bZ_m^n : z_i \in \bZ_m, 0 \leq i \leq t-1 \} \backslash \{b\}.
	\]

\section{Proof of \cref{thm:main}}
\label{sec:proof}

Let us recall Pyber's theorem on the existence of large abelian subgroup of a finite group.

\begin{theorem} [Pyber \cite{PY}]
	\label{thm:Pyber}
	There exists $\mu > 0$ such that every finite group $G$ contains an abelian subgroup $H$ of order at least $e^{\mu\sqrt{\log |G|}}$.
\end{theorem}

For a fixed $k \geq 1$, let $G$ be a finite group that has so large an order that contains an abelian subgroup $H$ with
\begin{equation}
\label{eq:n}
|H|\geq e^{\mu\sqrt{\log |G|}} \geq k^{3k \log k}.
\end{equation}
By the Fundamental Theorem of Finite Abelian Groups, $H$ is isomorphic to
\[
\bZ_{m_1} \times \cdots \times \bZ_{m_\tau},
\]
where $m_i$'s are prime powers, with $|H| = \prod_{i=1}^\tau m_i$. Let $\phi : \bZ_{m_1} \times \cdots \times \bZ_{m_\tau} \to H$ be an isomorphism.

By pigeonhole principle, there exist elements $\ell, r \in G$ such that
\begin{equation}
\label{eq:SlHr}
|S \cap (\ell H \times Hr)| \geq c|H|^2.
\end{equation}

Depending on the values $m_i$'s, we have 2 cases.

\textbf{If $m_i > k$ for some $i$,} reorder the cyclic groups if necessary, assume that $m_1 > k$. In this case, we set $\rho = 0$ and $t = k$. We show in this section that we can get a set of $4(k-1) < 4k$ elements of $G$ that spans at least $c(k-1)^2 \geq \frac c 4 k^2$ triples from $S$.

\textbf{If $m_i \leq k$ for all $i$,} reorder the cyclic groups if necessary, assume that $m_1 = m_2 = \cdots = m_\lambda = m \geq 2$ is the most popular index among the $m_i$'s and it occurs $\lambda$ times. In this case, we set integers $1 \leq t < m$ and $\rho \geq 1$ such that
\begin{equation}
\label{eq:k}
tm^\rho \leq k < (t+1)m^\rho,
\end{equation}
and we get a set of $4(tm^\rho-1) \leq 4k$ elements of $G$ that spans at least $c(tm^\rho-1)^2 \geq \frac c {16} k^2$ triples from $S$.

We note that in the second case, $\rho < \lambda$. By \eqref{eq:n}, we have
\[
k^{3k \log k} \leq |H| = \prod_{i=1}^\tau m_i \leq k^\tau,
\]
which implies that $\tau \geq 3k \log k \geq 2k\log k / \log 2 \geq 2k \log k / \log m$. Since there are at most $k$ possible distinct values of $m_i$'s, $\lambda \geq 2\log k / \log m > \log k / \log m$. On the other hand, \eqref{eq:k} implies that $k \geq m^\rho$, and so $\rho \leq \log k / \log m$.

Now, with the above chosen $\rho$ and $t$, we do the following. Let $\omega_1 = (1,0,\ldots,0) \in \bZ_{m_1} \times \cdots \times \bZ_{m_\tau}$, $\omega_2 = (0,1,0,\ldots,0) \in \bZ_{m_1} \times \cdots \times \bZ_{m_\tau}$, and so on. For $x \in H$, define
\[
A^{(1)}_{\ell x} := \{ \ell x \phi(\omega_1)^{\sigma_1} \cdots \phi(\omega_\rho)^{\sigma_\rho} \phi(\omega_{\rho+1})^\sigma : 0 \leq \sigma_i \leq m-1 \text{ and } 0 \leq \sigma \leq t-1 \}\backslash \{\ell x\}\subseteq \ell H,
\]
and
\[
A^{(2)}_{x r} := \{x \phi(\omega_1)^{\sigma_1} \cdots \phi(\omega_\rho)^{\sigma_\rho} \phi(\omega_{\rho+1})^\sigma r : 0 \leq \sigma_i \leq m-1 \text{ and } 0 \leq \sigma \leq t-1 \}\backslash \{xr\} \subseteq Hr.
\]

We now check the 4 corresponding conditions stated in \cref{sec:idea} one by one, as 4 lemmata.

\begin{lemma}
	\label{lem:condition1}
	For each $x \in H$, we have
	\[
	|A^{(1)}_{\ell x}| = |A^{(2)}_{x r}| = tm^\rho - 1.
	\]
\end{lemma}

\begin{proof}
	Since $\tau \geq \rho+1$, we have $m_i = m$ for all $1 \leq i \leq \rho+1$. Hence, the elements
	\[
	\{\sigma_i\omega_i \in \bZ_{m_1} \times \cdots \times \bZ_{m_\tau} : 1 \leq i \leq \rho+1 \text{ and } 0 \leq \sigma_i \leq m-1 \}
	\] are all distinct. This implies that the elements
	\[
	\{\phi(\omega_i)^{\sigma_i} \in H : 1 \leq i \leq \rho+1 \text{ and } 0 \leq \sigma_i \leq m-1 \}
	\]
	are distinct as well. With $t < m$, each of the sets $A^{(1)}_{\ell x}$ and $A^{(2)}_{x r}$ has size $tm^\rho - 1$.
\end{proof}

\begin{lemma}
	\label{lem:condition2}
	For $x,y \in H$, we have
	\[
	|A^{(1)}_{\ell x} A^{(2)}_{y r}| \leq 2(tm^\rho-1).
	\]
\end{lemma}

\begin{proof}
	Recall that $H$ is abelian. We write
	\[
	\begin{split}
	& \quad \,\, A^{(1)}_{\ell x} A^{(2)}_{y r} \\
	& \subseteq \{\ell x \phi(\omega_1)^{\sigma_1} \cdots \phi(\omega_\rho)^{\sigma_\rho} \phi(\omega_{\rho+1})^\sigma y \phi(\omega_1)^{\psi_1} \cdots \phi(\omega_\rho)^{\psi_\rho} \phi(\omega_{\rho+1})^\psi r: 0 \leq \sigma_i,\psi_j \leq m-1 \text{ and } 0 \leq \sigma,\psi \leq t-1 \} \\
	& =  \{\ell x \phi(\omega_1)^{\sigma_1+\psi_1} \cdots \phi(\omega_\rho)^{\sigma_\rho+\psi_\rho} \phi(\omega_{\rho+1})^{\sigma + \psi} y r : 0 \leq \sigma_i,\psi_j \leq m-1 \text{ and } 0 \leq \sigma,\psi \leq t-1 \} \\
	& =  \{\ell x \phi(\omega_1)^{\sigma_1} \cdots \phi(\omega_\rho)^{\sigma_\rho} \phi(\omega_{\rho+1})^{\sigma} y r : 0 \leq \sigma_i \leq m-1 \text{ and } 0 \leq \sigma \leq 2t-2 \},
	\end{split}
	\]
	which shows that the set $A^{(1)}_{\ell x} A^{(2)}_{x r} \subseteq \ell H r$ has size at most $(2t-1)m^\rho \leq 2(tm^\rho-1)$.
\end{proof}

The lemma below guarantees that the sets $B_{\ell x,yr}$ are distinct for different pairs of $(x,y) \in H^2$, where
\[
B_{\ell x,yr} := \{ (\ell a,b r,\ell abr) \in \ell H \times H r \times \ell H r : \ell a \in A^{(1)}_{\ell x}, br \in A^{(2)}_{y r} \}.
\]

\begin{lemma}
	\label{lem:condition3}
	\begin{enumerate}
		\item If $A^{(1)}_{\ell x} = A^{(1)}_{\ell y}$ for some $x,y \in H$, then $x = y$.
		\item If $A^{(2)}_{xr} = A^{(2)}_{yr}$ for some $x,y \in H$, then $x = y$.
	\end{enumerate}
\end{lemma}

\begin{proof}
	Note that $\ell x \not\in A_{\ell x}^{(1)}$. This allows us to recover $x$ from the set $A_{\ell x}^{(1)}$.
	
	Consider $\phi^{-1}(\ell^{-1}A_{\ell x}^{(1)}) \subseteq \bZ_{m_1} \times \cdots \times \bZ_{m_\tau}$, which is the same as
	\[
	\{ \phi^{-1}(x) + \sigma_1\omega_1 + \cdots + \sigma_\rho \omega_\rho + \sigma \omega_{\rho+1} : 0 \leq \sigma_i \leq m-1 \text{ and } 0 \leq \sigma \leq t-1 \} \backslash \{ \phi^{-1}(x) \}.
	\]
	Hence, the elements of the set $\phi^{-1}(\ell^{-1}A_{\ell x}^{(1)})$ as well as $\phi^{-1}(x)$ only differ in the first $\rho+1$ coordinates. If we consider only the $i$-coordinate, where $1 \leq i \leq \rho$, every but one element appears $tm^{\rho-1}$ times. This exceptional element from $\bZ_{m_i}$ is the $i$-th coordinate of $\phi^{-1}(x)$. Now, consider the $(\rho+1)$-coordinate. One of the elements appears only $m^\rho-1$ times. This is the $(\rho+1)$-coordinate of $\phi^{-1}(x)$.

	The proof of the second statement is similar.
\end{proof}

\begin{lemma}
	\label{lem:condition4}
	Given a triple $(\ell a,br,\ell abr) \in \ell H \times H r \times \ell H r$. The number of $B_{\ell x,yr}$'s which contains $(\ell a,br,\ell abr)$ is $(tm^\rho-1)^2$.
\end{lemma}

\begin{proof}
	We need to choose $x,y \in H$ such that $\ell a \in A^{(1)}_{\ell x}$ and $br \in A^{(2)}_{y r}$. To have $\ell a \in A^{(1)}_{\ell x}$, it is equivalent to have
	\[
	a \in \{ x \phi(\omega_1)^{\sigma_1} \cdots \phi(\omega_\rho)^{\sigma_\rho} \phi(\omega_{\rho+1})^\sigma : 0 \leq \sigma_i \leq m-1 \text{ and } 0 \leq \sigma \leq t-1 \} \backslash \{x\},
	\]
	which is in turn the same as choosing $x$ from the $(tm^\rho-1)$-element set
	\[
	\{ a \phi(\omega_1)^{-\sigma_1} \cdots \phi(\omega_\rho)^{-\sigma_\rho} \phi(\omega_{\rho+1})^{-\sigma} : 0 \leq \sigma_i \leq m-1 \text{ and } 0 \leq \sigma \leq t-1 \} \backslash \{a\}.
	\]
	Similarly, one can choose $y$ from the $(tm^\rho-1)$-element set \[
	\{ b \phi(\omega_1)^{-\sigma_1} \cdots \phi(\omega_\rho)^{-\sigma_\rho} \phi(\omega_{\rho+1})^{-\sigma} : 0 \leq \sigma_i \leq m-1 \text{ and } 0 \leq \sigma \leq t-1 \} \backslash \{b\}.
	\]
	Hence, there are $(tm^\rho-1)^2$ such $B_{\ell x,yr}$'s containing the given triple $(\ell a,br,\ell abr) \in \ell H \times H r \times \ell H r$.
\end{proof}

Finally, consider the pairs $((\ell a, br, \ell a b r), B_{\ell x,yr})$, where $(\ell a, br, \ell a b r) \in B_{\ell x,yr}$ are triples from $S$. By \eqref{eq:SlHr}, the number of triples from $S \cap (\ell H \times Hr)$ is at least $c|H|^2$. Hence, by \cref{lem:condition4}, the number of pairs we consider is at least $c|H|^2 (tm^\rho-1)^2$. There are $|H|^2$ different $B_{\ell x,yr}$'s by \cref{lem:condition3}. Using pigeonhole principle, there is a set $B_{\ell x,yr}$ which contains at least $c|H|^2 (tm^\rho-1)^2 / |H|^2 = c(tm^\rho-1)^2$ many triples from $S$. These $c(tm^\rho-1)^2$ triples are spanned by the set $A^{(1)}_{\ell x} \cup A^{(2)}_{y r} \cup A^{(1)}_{\ell x} A^{(2)}_{y r}$, which has at most
\[
|A^{(1)}_{\ell x}| + |A^{(2)}_{y r}| + |A^{(1)}_{\ell x} A^{(2)}_{y r}| \leq tm^\rho-1+tm^\rho-1+2(tm^\rho-1) = 4(tm^\rho-1)
\]
elements of $G$ by \cref{lem:condition1} and \cref{lem:condition2}, as desired.

Hence, we found $4k$ elements of $G$ that span at least $\frac c {16} k^2$ triples from $S$. By adjusting the constants, \cref{thm:main} is proved.

\section{Acknowledgements}
The author is grateful to J\'{o}zsef Solymosi for encouraging her to work on this problem and for some enlightening discussions.

\end{document}